\documentclass[11pt,a4paper]{article}
\usepackage{times}
\usepackage[T1]{fontenc}
\usepackage[utf8]{inputenc}
\usepackage{graphicx}
\usepackage{amsmath}
\usepackage{amsthm}
\usepackage{amsfonts}
\usepackage{caption}
\usepackage{float}
\usepackage{amssymb}
\usepackage{algorithm}
\usepackage{algpseudocode}
\usepackage{tikz}
\usepackage{pgfplots}
\usepackage{url}
\pgfplotsset{compat=1.18}
\newcommand{\ZZ}{Z\kern-.3em\raise-0.5ex\hbox{Z}}
\newcommand{\N}{\mathbb{N}}

\newcommand{\R}{\mathbb{R}}

\newcommand{\ra}{\right\rangle}
\newcommand{\la}{\left\langle}

\renewcommand{\H}{\mathcal{H}}

\newtheorem{assumption}{Assumption}[section]
\newtheorem{definition}[assumption]{Definition}
\newtheorem{proposition}[assumption]{Proposition}
\newtheorem{theorem}[assumption]{Theorem}
\newtheorem{lemma}[assumption]{Lemma}

\newenvironment{acknowledgement}
  {\section*{Acknowledgements}}
  {}
  
\title{On Symmetric Kernel Collocation for Nonlinear PDEs}

\author{
Milan Bacchetta\textsuperscript{*,1},
Tobias Ehring\textsuperscript{2},
and Bernard Haasdonk\textsuperscript{2}
\\[0.5em]
{\small
\textsuperscript{1}Mathematical Optimization and Data Science Group, Saarland University, Germany}
\\
\small{\textsuperscript{2}Institute of Applied Mathematics and Numerical Simulation, University of Stuttgart, Germany}
}

\date{\today}

\usepackage[backend=biber,url=true,doi=false,isbn=false,eprint=false,maxnames=99,]{biblatex}
\bibliography{non_lin_ker}
\begin{document}
\maketitle    

\begingroup
\renewcommand{\thefootnote}{*}
\footnotetext{Corresponding author: e-mail \texttt{milan.bacchetta@math.uni-sb.de}}
\endgroup

\begin{abstract}
This paper considers kernel-based approximation methods for nonlinear partial differential equations. To this end, the problem is formulated as an optimal-recovery generalized interpolation problem, that is, as an optimization problem in an RKHS with nonlinear functional constraints. This formulation provides the basis for a convergence analysis carried out directly in the RKHS and extends existing results by relaxing the uniqueness assumption on the PDE solution. In the nonunique case, the limiting object is characterized as a minimum-norm solution. Furthermore, a residual-greedy strategy for adaptive collocation point selection is proposed, and convergence of the resulting sequence of generalized interpolants is established. Numerical experiments for a stationary nonlinear heat equation illustrate the method and indicate that residual-greedy point selection can lead to markedly smaller PDE residuals than point sets selected according to fill-distance criteria.
\end{abstract}

\section{Introduction}

Symmetric kernel collocation is a meshless approach for approximating solutions of partial differential equations by functions from the native space of a positive definite kernel. In the linear case, the collocation equations lead to a linear generalized interpolation problem. For nonlinear equations this interpretation is no longer available directly, since imposing the equation at finitely many points leads to nonlinear conditions on the unknown function. To describe the class of nonlinear problems considered here, we separate the linear differential quantities from the nonlinear dependence on them. We consider boundary value problems of the form
\begin{align*}
\mathcal P(u)(x) &= f(x), \qquad x\in\Omega,\\
\mathcal B(u)(x) &= g(x), \qquad x\in\partial\Omega,
\end{align*}
where
\begin{align*}
\mathcal P(u)(x)
&=
\overline{P}\bigl(L_1^P u(x),\ldots,L_Q^P u(x)\bigr),\\
\mathcal B(u)(x)
&=
\overline{B}\bigl(L_1^B u(x),\ldots,L_R^B u(x)\bigr).
\end{align*}
Here, $\overline{P}$ and $\overline{B}$ are continuous, possibly nonlinear functions, while $L_i^P$ and $L_j^B$ are linear differential operators. This formulation keeps
the differential operations linear, while allowing the equations and boundary
conditions to depend nonlinearly on the resulting quantities. The precise
assumptions are stated below.

Linear symmetric kernel collocation can be interpreted as an optimal recovery
problem in the native space of the kernel. In this setting, the approximation
is sought among native-space functions satisfying the differential equation
and the boundary conditions at prescribed collocation points. For linear
differential and boundary operators, these conditions are linear functional
constraints. The nonlinear formulation studied in this work extends this
principle by admitting nonlinear pointwise constraints generated by the
operators $\mathcal P$ and $\mathcal B$. A natural consistency question is
whether the resulting collocation approximants converge to solutions of the
boundary value problem as the point sets are refined. We address this question
in two regimes: first, for sequences of collocation points with vanishing fill
distance, and second, for a novel residual-based, target-dependent greedy
strategy for selecting collocation points. In both regimes, convergence is
established in the native space of the kernel. Finally, numerical experiments for a stationary nonlinear heat-conduction problem
illustrate the proposed greedy strategy.

The analysis builds on reproducing kernel Hilbert spaces, kernel collocation for PDEs, optimal recovery, and greedy point selection for kernel methods. The RKHS perspective goes back to the foundational work of Aronszajn \cite{bib12}, while kernel collocation for PDEs using radial basis functions was studied, for instance, in \cite{bib13}. The nonlinear collocation framework considered here is based on the Gaussian-process formulation for nonlinear PDEs in \cite{bib1}; here we use the corresponding deterministic RKHS viewpoint. For nonlinear kernel collocation, \cite{bib2} obtained error estimates for families of point sets with vanishing fill distance, which is the refinement regime used in our first convergence result. Greedy point selection for kernel approximation has a long history, including adaptive greedy techniques for large RBF systems \cite{bib8}, data-independent point selection by geometric greedy methods \cite{bib9}, and convergence rates for the $P$-greedy algorithm \cite{bib14}. Related greedy ideas were later transferred to linear kernel collocation for PDEs in \cite{bib3,bib15}, and even to the parametric setting in \cite{bib6}. In contrast, the residual-greedy method considered here leads at each step to a nonlinear minimum-norm problem, so the existing linear greedy analysis does not apply directly. Nonlinear minimum-norm recovery problems in RKHSs also arise in optimal control; see \cite{bib5}.

\section{Preliminaries}
We first recall some basic facts on reproducing kernel Hilbert spaces and kernel collocation, following the structure of \cite{bib4}. Let $\Omega\subset\R^d$ be a bounded open set with closure $\overline{\Omega}$ and boundary $\partial \Omega$. A symmetric function $k:\overline \Omega\times \overline\Omega\to\R$ is called a positive definite kernel if, for any $x_1,\ldots,x_n\in\overline\Omega$, $n\in \N$, the associated kernel matrix
\begin{align*}
    K=(k(x_i,x_j))_{i,j=1}^n\in\R^{n\times n}
\end{align*}
is positive semi-definite. Every positive definite kernel $k$ corresponds to a unique reproducing kernel Hilbert space (RKHS), denoted by $\H_k(\overline\Omega)$. This is a Hilbert space with inner product $\la\cdot,\cdot\ra_{\H_k(\overline\Omega)}$ consisting of functions $f:\overline\Omega\to\R$ such that
\begin{itemize}
    \item for every $x\in\overline\Omega$, one has $k(\cdot,x)\in \H_k(\overline\Omega)$;
    \item for every $f\in \H_k(\overline\Omega)$, the reproducing property holds, namely
    \begin{align*}
        \la k(\cdot,x),f\ra_{\H_k(\overline\Omega)}=f(x).
    \end{align*}
\end{itemize}
A practical class of kernels is given by Sobolev kernels, among which Wendland kernels are a common example \cite{bib10}. For suitable choices of their smoothness parameters, these kernels generate RKHSs that are norm-equivalent to Sobolev spaces of integer order, with smoothness index greater than $d/2$. For our numerical experiments, we use the Gaussian kernel
\begin{equation}\label{eq:def_gaussian}
    k_\gamma (x,y):=\exp(-\gamma \|x-y\|_2^2),
    \qquad x,y\in \overline{\Omega},\quad \gamma>0.
\end{equation} 
The Gaussian kernel is widely used in practice due to its high accuracy and smoothness properties. Its native RKHS is not equivalent to a Sobolev space of finite order; rather, it consists of very smooth, in fact analytic, functions. For further details, see \cite[Sec.~4.4]{bib11}. If $k\in C^{2\ell}(\overline\Omega\times\overline\Omega)$ for some $\ell\in\N$, then the reproducing property extends to derivatives. In particular, $\H_k(\overline\Omega)\subset C^\ell(\overline\Omega)$ and, for every multi-index $\alpha\in\N_0^d$ with $|\alpha|\leq \ell$, one has
\begin{align*}
    \la \partial_2^\alpha k(\cdot,x),v\ra_{\H_k(\overline\Omega)}
    =\partial^\alpha v(x),
    \qquad v\in \H_k(\overline\Omega),\quad x\in\overline\Omega.
\end{align*}
Here, $\partial_2^\alpha$ denotes differentiation with respect to the second argument of $k$, and $\partial_1^\alpha$ is used analogously.

As a preliminary construction, we recall how the derivative-reproducing property leads to kernel collocation for linear PDEs, also known as generalized interpolation. Consider a linear PDE for an unknown $u\in C^\ell(\overline{\Omega})$ of the form
\begin{align*}
    L^P(u)(x)&=f(x),\quad x\in\Omega,\\
    L^B(u)(x)&=g(x),\quad x\in \partial\Omega.
\end{align*}
Here, $L^P$ and $L^B$ are prescribed linear differential operators of order at most $\ell$, that is, linear combinations of partial derivatives up to order $\ell$. The functions $f:\Omega\to\R$ and $g:\partial\Omega\to\R$ denote the given right-hand side and boundary data, respectively.

The collocation approach enforces the differential equations at selected points in the domain and on the boundary, and seeks a minimum-norm interpolant $u^\dagger$ that solves
\begin{equation}
\label{eq:lin_opt_recov}
\begin{aligned}
    \min_{u\in\mathcal{H}_k(\overline{\Omega})}
    \quad & \|u\|_{\mathcal{H}_k(\overline{\Omega})}^2 \\
    \text{subject to}\quad
    &L^P(u)(x_i)
    = f(x_i),
    && i\in [M], \\
    &L^B(u)(x_j')
    = g(x_j'),
    && j\in [M'] .
\end{aligned}
\end{equation}
Here and throughout, we write $[N]:=\{1,\ldots,N\}$ for $N\in\N$. Such a $u^\dagger$ can be explicitly computed. For a linear differential operator $L$ of order at most $\ell$ and a point $\bar{x}\in\overline{\Omega}$, we define the continuous functional $\delta_{\bar{x}}^L:\H_k(\overline{\Omega})\to\R$ by $\delta_{\bar{x}}^L(u):=L(u)(\bar{x})$. Let $\mathbf{X}=\{\delta_1,\ldots,\delta_n\}\subset \H_k(\overline{\Omega})'$ be a collection of such functionals, and let $\gamma_1,\ldots,\gamma_n\in\H_k(\overline{\Omega})$ denote their Riesz representers. We define the corresponding collocation matrix by
\begin{align*}
    K_{\mathbf{X}}:=\left(\delta_i(\gamma_j)\right)_{i,j=1}^n.
\end{align*}
For the PDE above, the relevant functionals are $\mathbf{X}=\{\delta_{x_1}^{L^P},\ldots,\delta_{x_{M}}^{L^P},\delta_{x_1'}^{L^B},\ldots,\delta_{x_{M'}'}^{L^B}\}$, where $x_1,\ldots,x_{M}\in\Omega$ and $x_1',\ldots,x_{M'}'\in\partial\Omega$ are interior and boundary collocation points, respectively. If the linear system
\begin{align*}
    K_{\mathbf{X}}\big((\alpha_i^P)_{i=1}^M,(\alpha^B_j)_{j=1}^{M'}\big)^T
    =
    \left(
    (f(x_i))_{i=1}^{M},
    (g(x_j'))_{j=1}^{M'}
    \right)^T
\end{align*}
admits a solution $\alpha\in\R^{M+M'}$, then the corresponding collocation approximant is given by
\begin{align*}
    u^\dagger
    =
    \sum_{i=1}^{M}\alpha_i^P\gamma_{x_i}^{L^P}
    +
    \sum_{i=1}^{M'}
    \alpha_{i}^B\gamma_{x_i'}^{L^B}.
\end{align*}
By construction, $u^\dagger$ satisfies the PDE at the prescribed collocation points. In concrete terms, both the entries of $K_{\mathbf{X}}$ and the Riesz representers are obtained by applying the corresponding differential operators to the appropriate arguments of the kernel. Here only, we use the notation $L_{(1)}$ and $L_{(2)}$ to indicate that the operator $L$ is applied to the first and second argument of the kernel, respectively. Thus one has $\gamma_{x_1}^{L^P}=L^P_{(2)} k(\cdot,x_1)$ and
\begin{align*}
    \delta_{x_1'}^{L^B}(\gamma_{x_1}^{L^P})
    =
    L_{(1)}^B L_{(2)}^P k(x_1',x_1).
\end{align*}
Expressions for these derivatives of the kernel can often be derived analytically beforehand, which allows for fast assembly of the associated linear system. For proofs and further details concerning symmetric kernel collocation for PDEs, we refer to \cite[Ch.~16]{bib4}.

\section{The nonlinear collocation scheme}

Linear symmetric kernel collocation can be interpreted as the construction of a minimum-norm generalized interpolant, as described in \eqref{eq:lin_opt_recov}. Following \cite{bib1}, this optimal-recovery perspective can be extended to a broad class of nonlinear PDEs. In this section, we introduce this class of problems and formulate the corresponding generalized interpolation problem. We seek a solution $u\in C^\ell(\overline{\Omega})$ such that all boundary differential operators appearing below are well defined on $\partial\Omega$, satisfying
\begin{equation}
\label{eq:pde}
\begin{split}
    \overline{P}\bigl(L_1^P u(x),\ldots,L_Q^P u(x)\bigr)
    &= f(x), \qquad x\in\Omega,\\
    \overline{B}\bigl(L_1^B u(x),\ldots,L_R^B u(x)\bigr)
    &= g(x), \qquad x\in\partial\Omega .
\end{split}
\end{equation}
Here, $\overline{P}:\R^Q\to\R$ and $\overline{B}:\R^R\to\R$ are continuous functions. The operators $L_i^P$, $i\in[Q]$, and $L_j^B$, $j\in[R]$, are linear differential operators of order at most $\ell\in\N$. The functions $f:\Omega\to\R$ and $g:\partial\Omega\to\R$ denote the given right-hand side and boundary data. We now choose a kernel $k\in C^{2\ell}(\overline{\Omega}\times\overline{\Omega})$ such that all relevant linear differential operators are well defined on $\H_k(\overline{\Omega})$. Following \cite{bib1}, we choose collocation points $x_1,\ldots,x_M\in\Omega$ and $x_1',\ldots,x_{M'}'\in\partial\Omega$, as in the linear setting. We then seek a function $u\in\H_k(\overline{\Omega})$ that satisfies the PDE and boundary conditions at the prescribed collocation points and has minimal RKHS norm. That is, we aim to solve
\begin{equation}
\label{eq:opt_recov}
\begin{aligned}
    \min_{u\in\H_k(\overline{\Omega})}
    \quad & \|u\|_{\H_k(\overline{\Omega})}^2 \\
    \text{subject to}\quad
    &\overline{P}\bigl(L_1^P u(x_i),\ldots,L_Q^P u(x_i)\bigr)
    = f(x_i),
    && i\in [M], \\
    &\overline{B}\bigl(L_1^B u(x_j'),\ldots,L_R^B u(x_j')\bigr)
    = g(x_j'),
    && j\in [M'] .
\end{aligned}
\end{equation}
To proceed, we will frequently refer to the following assumption, which collects the conditions needed both to ensure the existence of solutions to \eqref{eq:opt_recov} and to establish the convergence results below.
\begin{assumption}[Admissible collocation setting]
\label{ass:collocation-functionals}
Let the PDE \eqref{eq:pde} and a p.\,d.\ kernel
$k\in C^{2\ell}(\overline{\Omega}\times\overline{\Omega})$ be given.
Suppose that $\overline{\Omega}$ is compact, that \eqref{eq:pde} admits at
least one solution in $\H_k(\overline{\Omega})$, and that
$L_1^P,\ldots,L_Q^P$ and $L_1^B,\ldots,L_R^B$ are linear differential
operators of order at most $\ell$. For a given set of collocation points $X:=\{x_1,\ldots,x_M\}\cup\{x_1',\ldots,x_{M'}'\}$, where $x_1,\ldots,x_M\in\Omega$ and $x_1',\ldots,x_{M'}'\in\partial\Omega$, define
\begin{align*}
    \mathbf X(X)
    :=
    \Bigl\{
        \delta_{x_i}^{L_q^P}
        : i\in[M],\ q\in[Q]
    \Bigr\}
    \cup
    \Bigl\{
        \delta_{x_j'}^{L_r^B}
        : j\in[M'],\ r\in[R]
    \Bigr\}.
\end{align*}
The functionals in $\mathbf X(X)$ are required to be linearly independent.
\end{assumption}
Note that this assumption implies that $L_i^P:\H_k(\overline{\Omega})\to C(\overline{\Omega})$ for $i\in[Q]$ and $L_j^B:\H_k(\overline{\Omega})\to C(\Omega)$ for $j\in[R]$ are continuous; see \cite[Cor.~4.36]{bib11}.

Although \eqref{eq:opt_recov} is an infinite-dimensional constrained optimization problem, it can be reduced to a finite-dimensional problem restricted to the search to in the span of the Riesz representers associated with the collocation functionals in $\mathbf X(X)$. To this end, let $N:=M Q+M'R$, and define
\begin{equation*}
\begin{split}
    F : \R^N &\to \R^{M+M'}, \\
    (F(z))_m
    &:=
    \begin{cases}
        \overline{P}\big(
            z_{(m-1)Q+1},\, z_{(m-1)Q+2},\ldots,\, z_{mQ}
        \big),
        & m\in[M],\\[1ex]
        \overline{B}\big(
            z_{M Q+(j-1)R+1},\ldots,\,
            z_{M Q+jR}
        \big),
        & m=M+j,\ j\in[M'],
    \end{cases}
    \\[1ex]
    y_m
    &:=
    \begin{cases}
        f(x_m),
        & m\in[M],\\[0.5ex]
        g(x'_j),
        & m=M+j,\ j\in[M'].
    \end{cases}
\end{split}
\end{equation*}

\begin{proposition}[Adapted from {\cite[Thm.~3.6]{bib5}}]
\label{prop:finite_reduction}
Let $X=\{x_1,\ldots,x_M\}\cup\{x_1',\ldots,x_{M'}'\}$, with $x_1,\ldots,x_M\in\Omega$ and $x_1',\ldots,x_{M'}'\in\partial\Omega$, be a set of collocation points satisfying Assumption~\ref{ass:collocation-functionals}, and write $\mathbf X:=\mathbf X(X)$. Denote the Riesz representers of the corresponding collocation functionals by
\begin{align*}
    \gamma_{i,q}^P:=\bigl(\delta_{x_i}^{L_q^P}\bigr)^*,\quad i\in[M],\ q\in[Q],
    \qquad
    \gamma_{j,r}^B:=\bigl(\delta_{x_j'}^{L_r^B}\bigr)^*,\quad j\in[M'],\ r\in[R].
\end{align*}
Then $K_{\mathbf X}$ is invertible and the problem \eqref{eq:opt_recov} attains a minimum. Moreover, the square of the minimal norm is given by the optimal value of the finite-dimensional problem
\begin{equation}
\label{eq:quadratic_constrained}
\begin{aligned}
    \min_{z\in \R^N} \quad & z^\top K_{\mathbf X}^{-1} z \\
    \text{subject to} \quad & F(z)=y .
\end{aligned}
\end{equation}
If $z^\ast\in\R^N$ is a solution of \eqref{eq:quadratic_constrained} and $\alpha^\ast:=K_{\mathbf X}^{-1}z^\ast$, then the corresponding minimizer of \eqref{eq:opt_recov} is given by
\begin{align*}
    u^\ast
    &=
    \sum_{i=1}^{M}\sum_{q=1}^{Q}
    \alpha^\ast_{(i-1)Q+q}\,
    \gamma_{i,q}^P  +
    \sum_{j=1}^{M'}\sum_{r=1}^{R}
    \alpha^\ast_{MQ+(j-1)R+r}\,
    \gamma_{j,r}^B .
\end{align*}
\end{proposition}

\section{Convergence analysis of the optimal recovery scheme}

This section studies how the optimal recovery method behaves as the number of collocation points increases. We first prove an abstract convergence result showing that convergence of the PDE residuals provides a criterion for convergence of a sequence of generalized interpolants to the set of PDE solutions. Although this result is stated as a proposition, it should be viewed as one of the main analytical tools of the paper and contains a technically substantial part of the convergence analysis. We then consider two concrete settings: convergence under a diminishing fill distance condition, and convergence for a residual-greedy, target-dependent collocation point selection method. Throughout the section, we study a sequence of collocation point sets $X_n\subset\overline{\Omega}$ and the corresponding minimum-norm generalized interpolants $u_n\in\H_k(\overline{\Omega})$. The goal is to prove strong convergence directly in the RKHS, in contrast to \cite{bib1}, where compact embedding into a weaker space is assumed and convergence is established in the corresponding weaker topology.

If the PDE does not have a unique solution in the RKHS, convergence can still be established toward the set of minimum-norm RKHS solutions. To this end, we define the set of minimum-norm solutions by
\begin{align*}
    S_k
    :=
    \operatorname*{argmin}_{\substack{
        u\in\H_k(\overline{\Omega}),\\
        u \text{ solves } \eqref{eq:pde}
    }}
    \|u\|_{\H_k(\overline{\Omega})}.
\end{align*}
The dependence on $k$ is emphasized because the choice of kernel determines the solution space and, hence, the solutions selected by the minimum-norm criterion. Moreover, since Assumption~\ref{ass:collocation-functionals} ensures the existence of a solution to \eqref{eq:pde} in $\H_k(\overline{\Omega})$, Lemma~\ref{lem:min-norm-solutions-nonempty} in the appendix guarantees that $S_k\neq\emptyset$.
    
In this section, we will make use of the PDE and boundary residuals defined by
\begin{align*}
    R_P[u](x)
    &:=
    \left|
    \overline{P}\bigl(L_1^P u(x),\ldots,L_Q^P u(x)\bigr)
    - f(x)
    \right|,
    \qquad x\in\Omega,\\
    R_B[u](x)
    &:=
    \left|
    \overline{B}\bigl(L_1^B u(x),\ldots,L_R^B u(x)\bigr)
    - g(x)
    \right|,
    \qquad x\in\partial\Omega .
\end{align*}
We proceed in two steps. First, we prove the general statement that pointwise
convergence of these residuals along the sequence of optimal recovery
interpolants $(u_n)_{n\in\N}$, namely
\begin{align*}
    R_P[u_n](x)&\to 0
    \qquad\text{for all }x\in\Omega,\\
    R_B[u_n](x)&\to 0
    \qquad\text{for all }x\in\partial\Omega,
\end{align*}
is sufficient to obtain convergence of the interpolants to the set of PDE
solutions. Second, we establish this residual convergence for a sequence generated by specific collocation sets. The following proposition
provides the first step.

\begin{proposition}
\label{prop:abstract_convergence}
Let $(X_n)_{n\in\N}$ be a sequence of collocation point sets satisfying Assumption~\ref{ass:collocation-functionals}, where
\begin{align*}
    X_n
    &:=
    X_{\Omega,n}\cup X_{\partial\Omega,n},\\
    X_{\Omega,n}
    &:=
    \{x_1^n,\ldots,x_{M_n}^n\}\subset\Omega,\\
    X_{\partial\Omega,n}
    &:=
    \{x_1^{n\prime},\ldots,x_{M'_n}^{n\prime}\}
    \subset\partial\Omega .
\end{align*}
For each $n\in\N$, let $u_n\in\H_k(\overline{\Omega})$ be a minimizer of \eqref{eq:opt_recov} corresponding to the collocation set $X_n$. Assume that
\begin{align*}
    R_P[u_n](x)&\to 0
    \qquad\text{for all }x\in\Omega,\\
    R_B[u_n](x)&\to 0
    \qquad\text{for all }x\in\partial\Omega .
\end{align*}
Then the following statements hold:
\begin{enumerate}
    \item[\textnormal{a)}] The sequence $(u_n)_{n\in\N}$ converges to the set $S_k$ of minimum-norm solutions in the sense that
    \begin{align*}
        \operatorname{dist}_{\H_k(\overline{\Omega})}(u_n,S_k)\to 0.
    \end{align*}

    \item[\textnormal{b)}] If the minimum-norm solution of \eqref{eq:pde} is unique, say $u^\star$, then the full sequence converges strongly,
    \begin{align*}
        u_n\to u^\star
        \qquad\text{in }\H_k(\overline{\Omega}).
    \end{align*}
\end{enumerate}
\end{proposition}

\begin{proof}
By Lemma~\ref{lem:min-norm-solutions-nonempty}, the set $S_k$ is nonempty. Choose a minimum-norm solution $\bar u\in S_k$. Since $\bar u$ solves \eqref{eq:pde}, it is admissible for every collocation problem \eqref{eq:opt_recov}. Hence, by the optimality of $u_n$,
\begin{align*}
    \|u_n\|_{\H_k(\overline{\Omega})}\leq \|\bar u\|_{\H_k(\overline{\Omega})}
    \qquad\text{for all }n\in\N.
\end{align*}
Thus $(u_n)_{n\in\N}$ is bounded in $\H_k(\overline{\Omega})$. Consider any weakly convergent subsequence $u_{n_j}\rightharpoonup u^\dagger$ in $\H_k(\overline{\Omega})$. For fixed $x\in\Omega$ and $q\in[Q]$, the map $u\mapsto L_q^P u(x)$ is a continuous linear functional by Assumption~\ref{ass:collocation-functionals}. Therefore $L_q^P u_{n_j}(x)\to L_q^P u^\dagger(x)$ for every $q\in[Q]$. By the continuity of $\overline P$, it follows that $R_P[u_{n_j}](x)\to R_P[u^\dagger](x)$. Since $R_P[u_n](x)\to0$ by assumption, we obtain $R_P[u^\dagger](x)=0$ for all $x\in\Omega$. The boundary residual is treated analogously, and hence $u^\dagger$ solves \eqref{eq:pde}.

Since $u^\dagger$ is a solution, its norm is bounded below by $\|\bar u\|_{\H_k(\overline{\Omega})}$, the minimal norm among all solutions. Together with weak lower semicontinuity of the norm and the estimate above, this yields
\begin{align*}
    \|\bar u\|_{\H_k(\overline{\Omega})}
    \leq
    \|u^\dagger\|_{\H_k(\overline{\Omega})}
    \leq
    \liminf_{j\to\infty}\|u_{n_j}\|_{\H_k(\overline{\Omega})}
    \leq
    \limsup_{j\to\infty}\|u_{n_j}\|_{\H_k(\overline{\Omega})}
    \leq
    \|\bar u\|_{\H_k(\overline{\Omega})}.
\end{align*}
Consequently, $u^\dagger\in S_k$ and $\|u_{n_j}\|_{\H_k(\overline{\Omega})}\to\|u^\dagger\|_{\H_k(\overline{\Omega})}$. Since weak convergence together with convergence of norms implies strong convergence in a Hilbert space, we have
\begin{align*}
    u_{n_j}\to u^\dagger
    \qquad\text{in }\H_k(\overline{\Omega}).
\end{align*}

We now prove a). Suppose, for contradiction, that $\operatorname{dist}_{\H_k(\overline{\Omega})}(u_n,S_k)\not\to0$. Then there exist $\varepsilon>0$ and a subsequence $(u_{n_j})_{j\in\N}$ such that
\begin{align*}
    \operatorname{dist}_{\H_k(\overline{\Omega})}(u_{n_j},S_k)
    \geq \varepsilon
    \qquad\text{for all }j\in\N.
\end{align*}
By boundedness and reflexivity, this subsequence has a weakly convergent further subsequence. By the preceding argument, this further subsequence converges strongly to some $u^\dagger\in S_k$. Hence its distance to $S_k$ converges to zero, contradicting the estimate above. Thus
\begin{align*}
    \operatorname{dist}_{\H_k(\overline{\Omega})}(u_n,S_k)\to0.
\end{align*}
If $S_k=\{u^\star\}$, then b) follows directly from a), since
\begin{align*}
    \operatorname{dist}_{\H_k(\overline{\Omega})}(u_n,S_k)
    =
    \|u_n-u^\star\|_{\H_k(\overline{\Omega})}.
\end{align*}
This proves the claim.
\end{proof}

\subsection{Convergence for small fill distance}

This subsection applies Proposition~\ref{prop:abstract_convergence} to sequences of collocation point sets with vanishing fill distance. Such sequences can be generated, for example, by using grids of decreasing mesh width to fill the domain uniformly, or by constructing collocation points via the geometric greedy method described in \cite[Sec.~5]{bib9}. In kernel interpolation, convergence is commonly formulated in terms of a vanishing fill distance. In the present PDE setting, however, the distribution of collocation points has to be controlled both in the interior of the domain and on the boundary. This motivates the following definition.

\begin{definition}[Effective fill distance]
Let $\Gamma\subset\R^d$ and let $Y\subset\Gamma$ be nonempty and finite. The fill distance of $Y$ in $\Gamma$ is
\begin{align*}
    h_Y^\Gamma
    :=
    \sup_{x\in\Gamma}\inf_{y\in Y}\|x-y\|.
\end{align*}
Let $\Omega\subset\R^d$ be open and let $X=X_\Omega\cup X_{\partial\Omega}$, where $X_\Omega\subset\Omega$ and $X_{\partial\Omega}\subset\partial\Omega$ are nonempty and finite. The effective fill distance of $X$ with respect to $\Omega$ is given by
\begin{align*}
    \overline h_X^\Omega
    :=
    \max\{h_{X_\Omega}^{\Omega},h_{X_{\partial\Omega}}^{\partial\Omega}\}.
\end{align*}
\end{definition}
Proposition~\ref{prop:abstract_convergence} yields the following convergence result for collocation sets with vanishing effective fill distance. Compared with \cite{bib1}, this yields a meaningful result without assuming uniqueness of the PDE solution.
\begin{theorem}
\label{thm:convergence_fill_distance}
Assume the setting of Assumption~\ref{ass:collocation-functionals}. Let $(X_n)_{n\in\N}$ be a sequence of sets of distinct collocation points of the form
\begin{align*}
    X_n
    &:=
    X_{\Omega,n}\cup X_{\partial\Omega,n},\\
    X_{\Omega,n}
    &:=
    \{x_1^n,\ldots,x_{M_n}^n\}\subset\Omega,\\
    X_{\partial\Omega,n}
    &:=
    \{x_1^{n\prime},\ldots,x_{M'_n}^{n\prime}\}
    \subset\partial\Omega .
\end{align*}
Assume that each $X_n$ satisfies the linear independence condition in Assumption~\ref{ass:collocation-functionals}, and write $\mathbf X_n:=\mathbf X(X_n)$. Moreover, assume that the effective fill distance satisfies
\begin{align*}
    \lim_{n\to\infty}\overline h_{X_n}^\Omega
    =
    0 .
\end{align*}
For each $n\in\N$, let $u_n\in\H_k(\overline{\Omega})$ be a minimizer of \eqref{eq:opt_recov} corresponding to the collocation set $X_n$. Then:
\begin{enumerate}
    \item[\textnormal{a)}] The sequence $(u_n)_{n\in\N}$ converges to the set $S_k$ of minimum-norm solutions in the sense that
    \begin{align*}
        \operatorname{dist}_{\H_k(\overline{\Omega})}(u_n,S_k)\to 0.
    \end{align*}

    \item[\textnormal{b)}] If \eqref{eq:pde} admits a unique minimum-norm solution $u^\star\in\H_k(\overline{\Omega})$, then
    \begin{align*}
        u_n
        \to
        u^\star
        \qquad
        \text{in } \H_k(\overline{\Omega}) .
    \end{align*}
\end{enumerate}
\end{theorem}
\begin{proof}
Let $u^\dagger\in\H_k(\overline{\Omega})$ be any solution of \eqref{eq:pde}, whose existence is guaranteed by Assumption~\ref{ass:collocation-functionals}. Since $u^\dagger$ satisfies the equation and the boundary condition at every collocation point, it is feasible for \eqref{eq:opt_recov} for each $n\in\N$. Hence, by the minimality of $u_n$,
\begin{align*}
    \|u_n\|_{\H_k(\overline{\Omega})}
    \leq
    \|u^\dagger\|_{\H_k(\overline{\Omega})}
    \qquad\text{for all }n\in\N.
\end{align*}
Thus $(u_n)_{n\in\N}$ is bounded in $\H_k(\overline{\Omega})$. For $n\in\N$, define
\begin{align*}
    \mathcal P_n(x)
    &:=
    \overline{P}\bigl(
        L_1^P u_n(x),\ldots,L_Q^P u_n(x)
    \bigr),
    \qquad x\in\overline{\Omega},\\
    \mathcal B_n(x)
    &:=
    \overline{B}\bigl(
        L_1^B u_n(x),\ldots,L_R^B u_n(x)
    \bigr),
    \qquad x\in\partial\Omega .
\end{align*}
By Lemma~\ref{lem:equicontinuity}, the families $\{\mathcal P_n:n\in\N\}$ and $\{\mathcal B_n:n\in\N\}$ are uniformly equicontinuous on $\overline{\Omega}$ and $\partial\Omega$, respectively.

We first show that the interior residuals vanish pointwise. Fix $x\in\Omega$ and let $\varepsilon>0$. By the uniform equicontinuity of $\{\mathcal P_n:n\in\N\}$ and the continuity of $f$ at $x$, there exists $\delta>0$ such that, for all $n\in\N$ and all $\overline{x}\in\Omega$ with $\|x-\overline{x}\|<\delta$,
\begin{align*}
    |\mathcal P_n(x)-\mathcal P_n(\overline{x})|
    <
    \frac{\varepsilon}{2},
    \qquad
    |f(x)-f(\overline{x})|
    <
    \frac{\varepsilon}{2}.
\end{align*}
Since $\overline h_{X_n}^\Omega\to0$, there exists $n^\star\in\N$ such that, for every $n\geq n^\star$, one can find $x_i^n\in X_{\Omega,n}$ with $\|x-x_i^n\|<\delta$. The collocation constraints imply $\mathcal P_n(x_i^n)=f(x_i^n)$, and therefore, for all $n\geq n^\star$,
\begin{align*}
    R_P[u_n](x)
    &=
    |\mathcal P_n(x)-f(x)|\\
    &=
    \left|
        \mathcal P_n(x)-f(x)
        -
        \bigl(\mathcal P_n(x_i^n)-f(x_i^n)\bigr)
    \right|\\
    &\leq
    |\mathcal P_n(x)-\mathcal P_n(x_i^n)|
    +
    |f(x)-f(x_i^n)|<\varepsilon .
\end{align*}
Hence $R_P[u_n](x)\to0$ for every $x\in\Omega$. 

The boundary residuals are handled analogously, yielding $R_B[u_n](x)\to0$ for every $x\in\partial\Omega$. Hence the hypotheses of Proposition~\ref{prop:abstract_convergence} are satisfied. Statement \textnormal{a)} then follows directly, while statement \textnormal{b)} follows under the additional uniqueness assumption.
\end{proof}

\subsection{The residual-greedy approach}

Residual-greedy strategies have been successfully applied to linear symmetric kernel collocation. They provide a highly target-dependent method for choosing collocation points. In particular, the convergence results in \cite{bib15} motivate residual-based greedy strategies for nonlinear symmetric kernel collocation. In the nonlinear setting considered here, we prove convergence of the resulting adaptive approximation scheme without deriving convergence rates. The numerical experiments reported later in the paper suggest that this strategy improves convergence behaviour and stability, and that it can provide accurate approximations in practice even in situations where the exact solution may not belong to the chosen RKHS.

We begin by formulating the residual-greedy scheme, which can be viewed as a natural extension of the strategy presented in \cite{bib3}. The residual serves as an error indicator. For $u\in\H_k(\overline\Omega)$, define
\begin{align*}
    \eta_u:\overline{\Omega}\to [0,\infty),\qquad
    x\mapsto
    \begin{cases}
        R_P[u](x), & x\in\Omega,\\
        R_B[u](x), & x\in\partial\Omega.
    \end{cases}
\end{align*}
Since $\eta_u$ is not necessarily continuous when approaching the boundary from the interior, it need not attain a maximum on $\overline{\Omega}$. Nevertheless, in the setting of Assumption~\ref{ass:collocation-functionals}, it is bounded. Indeed, since $R_P[u]$ extends continuously to $\overline{\Omega}$, one has
\begin{align*}
    \sup_{x\in\overline\Omega} \eta_u(x)
    \leq
    \sup_{x\in\overline\Omega} R_P[u](x)
    +
    \sup_{x\in\partial\Omega} R_B[u](x)
    <\infty.
\end{align*}
This motivates the use of a weak greedy selection rule.

\begin{definition}[$\alpha$-weak greedy sequence]
\label{def:weak-greedy-sequence}
In the setting of Assumption~\ref{ass:collocation-functionals}, let $\eta_u$ be defined as above, and let $\alpha\in(0,1)$. A sequence $(x_n)_{n\in\N}\subset\overline{\Omega}$ is called an $\alpha$-weak greedy sequence if it can be constructed iteratively as follows. Set $u_0\equiv 0\in\H_k(\overline{\Omega})$. For $n\geq 1$:
\begin{itemize}
    \item select $x_n\in\overline{\Omega}\setminus\{x_1,\ldots,x_{n-1}\}$ such that
    \begin{align*}
        \eta_{u_{n-1}}(x_n)
        \geq
        \alpha\sup_{x\in\overline{\Omega}}\eta_{u_{n-1}}(x);
    \end{align*}

    \item compute $u_n$ as a minimizer of \eqref{eq:opt_recov} using $x_1,\ldots,x_n$ as interior or boundary collocation points according to whether they lie in $\Omega$ or $\partial\Omega$.
\end{itemize}
\end{definition}
The sequence is treated as infinite for the purposes of the analysis; if an exact solution is reached after finitely many steps, the remaining points may be chosen arbitrarily among those not selected before. In computations, by contrast, the iteration is stopped once the residual reaches a prescribed tolerance in the $\infty$-norm or after a prescribed maximum number of steps. A further practical issue is the existence of a minimizer $u_n$ at each iteration. In certain RKHS settings, such existence can be guaranteed by results such as \cite[Thm.\,10.45]{bib4}. To keep the framework general, this existence requirement is included in the definition of the greedy sequence. Under the same assumptions as in the preceding convergence results, we obtain the following theorem.
\begin{theorem}[Convergence of the residual-greedy scheme]
\label{thm:greedy_convergence}
Assume the admissible collocation setting of Assumption~\ref{ass:collocation-functionals}. Let $\alpha\in(0,1)$ and let $(x_n)_{n\in\N}\subset\overline{\Omega}$ be an $\alpha$-weak greedy sequence with corresponding generalized interpolants $(u_n)_{n\in\N}\subset\H_k(\overline{\Omega})$. Then:
\begin{enumerate}
    \item[\textnormal{a)}] The sequence $(u_n)_{n\in\N}$ converges to the set $S_k$ of minimum-norm solutions in the sense that
    \begin{align*}
        \operatorname{dist}_{\H_k(\overline{\Omega})}(u_n,S_k)\to 0.
    \end{align*}

    \item[\textnormal{b)}] If \eqref{eq:pde} admits a unique minimum-norm solution $u^\star\in\H_k(\overline{\Omega})$, then
    \begin{align*}
        u_n\to u^\star
        \qquad\text{in }\H_k(\overline{\Omega}).
    \end{align*}
\end{enumerate}
\end{theorem}
\begin{proof}
Let $\bar u\in S_k$ be a minimum-norm solution of \eqref{eq:pde}, whose existence follows from Assumption \ref{ass:collocation-functionals} and Lemma \ref{lem:min-norm-solutions-nonempty}. Since $\bar u$ is feasible for every collocation problem generated by the greedy scheme, the optimality of $u_n$ gives
\begin{align*}
    \|u_n\|_{\H_k(\overline{\Omega})}
    \leq
    \|\bar u\|_{\H_k(\overline{\Omega})}
    \qquad\text{for all }n\in\N.
\end{align*}
Thus $(u_n)_{n\in\N}$ is bounded in $\H_k(\overline{\Omega})$. Moreover, since \eqref{eq:pde} admits a solution in $\H_k(\overline{\Omega})$, the right-hand side $f$ admits a continuous extension to $\overline{\Omega}$ through the left-hand side of the PDE. Hence, by Lemma~\ref{lem:equicontinuity} and the uniform continuity of the continuous extensions of the data, the residual families $(R_P[u_n])_{n\in\N}$ and $(R_B[u_n])_{n\in\N}$ are uniformly equicontinuous on $\overline{\Omega}$ and $\partial\Omega$, respectively.

We claim that $\|\eta_{u_n}\|_\infty\to0$. Suppose, for contradiction, that this is not the case. Then there exist $\varepsilon>0$ and a subsequence $(u_{n_p})_{p\in\N}$ such that $\|\eta_{u_{n_p}}\|_\infty\geq\varepsilon$ for all $p\in\N$. By uniform equicontinuity, there exists $\delta>0$ such that, for all $n\in\N$,
\begin{align*}
    |R_P[u_n](x)-R_P[u_n](z)|
    <
    \alpha\varepsilon
    \qquad
    \text{whenever }x,z\in\overline{\Omega},\ \|x-z\|<\delta,
\end{align*}
and
\begin{align*}
    |R_B[u_n](x)-R_B[u_n](z)|
    <
    \alpha\varepsilon
    \qquad
    \text{whenever }x,z\in\partial\Omega,\ \|x-z\|<\delta.
\end{align*}
Set $\xi_p:=x_{n_p+1}$. By the $\alpha$-weak greedy choice, $\eta_{u_{n_p}}(\xi_p)\geq\alpha\varepsilon$. We show that the points $(\xi_p)_{p\in\N}$ are $\delta$-separated within $\Omega$ and within $\partial\Omega$. Indeed, let $r<p$. Then $\xi_r=x_{n_r+1}$ has already been selected before the construction of $u_{n_p}$ and is therefore one of the collocation points used for $u_{n_p}$. If $\xi_p,\xi_r\in\Omega$, then the collocation constraint gives $R_P[u_{n_p}](\xi_r)=0$. Hence $\|\xi_p-\xi_r\|<\delta$ would imply
\begin{align*}
    \alpha\varepsilon
    \leq
    R_P[u_{n_p}](\xi_p)
    =
    |R_P[u_{n_p}](\xi_p)-R_P[u_{n_p}](\xi_r)|
    <
    \alpha\varepsilon,
\end{align*}
which is impossible. The same argument on $\partial\Omega$ shows that $\|\xi_p-\xi_r\|\geq\delta$ whenever $\xi_p,\xi_r\in\partial\Omega$. Thus either infinitely many of the points $\xi_p$ lie in $\Omega$, or infinitely many lie in $\partial\Omega$. Since $\overline{\Omega}$, and hence also $\partial\Omega$, is compact, such an infinite subsequence has an accumulation point. This contradicts the $\delta$-separation established above. Therefore
\begin{align*}
    \|\eta_{u_n}\|_\infty\to0.
\end{align*}
In particular, $R_P[u_n](x)\to0$ for every $x\in\Omega$ and $R_B[u_n](y)\to0$ for every $y\in\partial\Omega$. The assertions now follow directly from Proposition~\ref{prop:abstract_convergence}.
\end{proof}
\section{Numerical experiments}

We focus on the behaviour of the residual-greedy method. As a model problem, we consider the stationary nonlinear heat equation on $\Omega=(0,1)^2$,
\begin{align*}
    \Delta u(x)+u^3(x)
    &=
    f(x),
    \qquad x\in\Omega,\\
    u(x)
    &=
    g(x),
    \qquad x\in\partial\Omega .
\end{align*}
This problem was also considered in \cite{bib1}. Following the approach proposed there, we solve the corresponding optimal recovery problems \eqref{eq:opt_recov} by a Gauss--Newton method. The data are prescribed through a chosen exact solution $u$: we set $f(x):=\Delta u(x)+u^3(x)$ for $x\in\Omega$ and $g(x):=u(x)$ for $x\in\partial\Omega$. All experiments use the Gaussian kernel defined in \eqref{eq:def_gaussian} with shape parameter $\gamma=5$.

To test the method under the theoretical assumptions, we use
$u(x):=u_{\mathcal H}(x):=k_\gamma((0.2,0.5),x)$ for
$x\in\overline{\Omega}$. We also consider the exact solution $u(x):=u_{\sin}(x):=\sin(\pi x_1)\sin(\pi x_2)$ for
$x\in\overline{\Omega}$, which, for the Gaussian kernel, does not belong to
the associated RKHS; see, for instance, \cite[Lem.~A.3.1]{bib7}.

Our implementation differs from the theoretical framework in two respects. First, instead of selecting each new collocation point from the full set $\overline{\Omega}$, we periodically alternate between choosing interior and boundary collocation points; in our experiments, this led to more stable behaviour. Second, the maximization of the residual is carried out over fixed discrete grids $\Omega_G$ and $\partial\Omega_G$. This is computationally natural, since the residual is generally nonconvex and global maximization is not practical, especially in the interior, where a maximum need not be attained. The resulting implementation is summarized in Fig.~\ref{alg:residual_greedy}. The optimal recovery problem in $\textsc{OptRecov}$ is solved by the Gauss--Newton scheme proposed in \cite{bib1}.

\begin{algorithm}[ht]
\captionsetup{margin=4cc}
\caption{Residual-greedy method.}
\label{alg:residual_greedy}
\footnotesize
\centering
\begin{minipage}{0.92\textwidth}

\textbf{Input:} Interior grid $\Omega_G$, boundary grid $\partial\Omega_G$,
maximum number of iterations $\mathrm{MaxIter}$.

\textbf{Output:} Approximation $u_{\mathrm{MaxIter}}$.

\textbf{Initialize} $\mathcal X_{\Omega,0}:=\emptyset$,
$\mathcal X_{\partial\Omega,0}:=\emptyset$, and $u_0\equiv0$.

\textbf{for} $n=1,\ldots,\mathrm{MaxIter}$:

\hspace*{1em}Set $\mathcal X_{\Omega,n}:=\mathcal X_{\Omega,n-1}$ and
$\mathcal X_{\partial\Omega,n}:=\mathcal X_{\partial\Omega,n-1}$.

\hspace*{1em}\textbf{if} $n\bmod 4=0$:

\hspace*{2em}Choose
\begin{align*}
    x_n\in\arg\max_{x\in\partial\Omega_G\setminus\mathcal X_{\partial\Omega,n}}
    R_B[u_{n-1}](x),
\end{align*}

\hspace*{2em}and set
$\mathcal X_{\partial\Omega,n}:=
\mathcal X_{\partial\Omega,n}\cup\{x_n\}$.

\hspace*{1em}\textbf{else}:

\hspace*{2em}Choose
\begin{align*}
    x_n\in\arg\max_{x\in\Omega_G\setminus\mathcal X_{\Omega,n}}
    R_P[u_{n-1}](x),
\end{align*}

\hspace*{2em}and set
$\mathcal X_{\Omega,n}:=
\mathcal X_{\Omega,n}\cup\{x_n\}$.

\hspace*{1em}Compute
\begin{align*}
    u_n:=\textsc{OptRecov}(\mathcal X_{\Omega,n},
    \mathcal X_{\partial\Omega,n}).
\end{align*}

\textbf{return} $u_{\mathrm{MaxIter}}$.

\end{minipage}
\end{algorithm}

As a baseline, we compare the residual-greedy method with collocation sets generated by the geometric greedy method from \cite[Sec.~5]{bib9}. The geometric greedy procedure is applied separately to the interior grid $\Omega_G$ and the boundary grid $\partial\Omega_G$, producing two ordered sets of candidate points. To ensure a fair comparison, the baseline uses the same number of interior and boundary collocation points as the residual-greedy construction at each iteration, corresponding to the same $3:1$ ratio of interior to boundary points. All reported residuals and errors are evaluated on fixed validation grids.

For the test solution $u_\H$, the residual-greedy method consistently yields smaller residual errors than the approach based on precomputed collocation points; see Figure~\ref{fig:residual-comparison-k-g5-500}. Moreover, the benefit of adding further precomputed points appears to diminish after approximately 250 centers, whereas the residual-greedy approach continues to improve and eventually reaches an $\infty$-norm of the interior residual that is several orders of magnitude lower. 

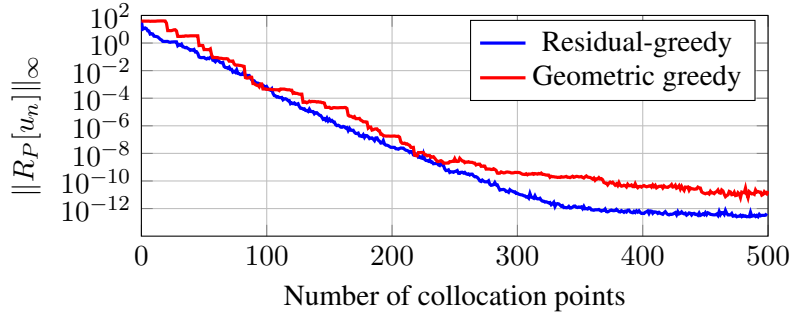
\begin{figure}[hb]
\centering
\begin{tikzpicture}
\begin{axis}[
    width=0.78\textwidth,
    height=45mm,
    xlabel={Number of collocation points},
    ylabel={$\|R_P[u_n]\|_\infty$},
    ymode=log,
    grid=both,
    ymax=1e2,
    ymin=1e-14,
    ytick={1e-12,1e-10,1e-8,1e-6,1e-4,1e-2,1e0,1e2},
    minor x tick num=0,
    every axis plot/.append style={line width=1.2pt},
    xmin=0,
    xmax=500,
    legend style={
        at={(0.98,0.98)},
        anchor=north east,
        legend columns=1
    },
]
\addplot+[mark=none] table[x=x,y=residual] {data/residual_fgreedy_g5.dat};
\addlegendentry{Residual-greedy}

\addplot+[mark=none] table[x=x,y=residual] {data/residual_fixed_g5.dat};
\addlegendentry{Geometric greedy}
\end{axis}
\end{tikzpicture}
\caption{Comparison of the $\infty$-norm of the interior residual for $u_\H$ using residual-greedy and geometric greedy collocation points.}
\label{fig:residual-comparison-k-g5-500}
\end{figure}
Turning to Figure~\ref{fig:error-comparison-k-g5-500}, the advantage is less pronounced. Although the residual-greedy approach slightly outperforms the predetermined centers for most iteration counts, its main benefit is that it stabilizes at a more accurate approximation after approximately 250 iterations.

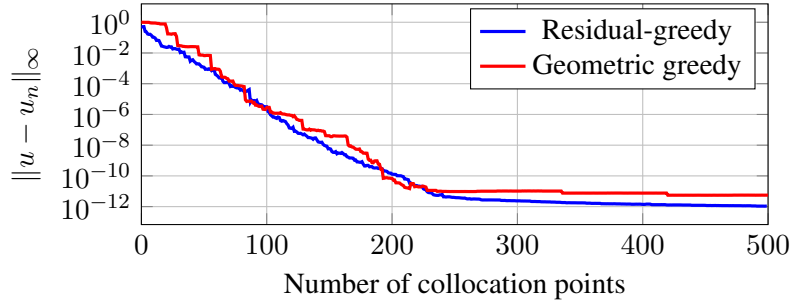
\begin{figure}[H]
\centering
\begin{tikzpicture}
\begin{axis}[
    width=0.78\textwidth,
    height=45mm,
    xlabel={Number of collocation points},
    ylabel={$\|u-u_n\|_\infty$},
    ymode=log,
    grid=both,
    ytick={1e-12,1e-10,1e-8,1e-6,1e-4,1e-2,1e0},
    minor x tick num=0,
    every axis plot/.append style={line width=1.2pt},
    xmin=0,
    xmax=500,
    legend style={
        at={(0.98,0.98)},
        anchor=north east,
        legend columns=1
    },
]
\addplot+[mark=none] table[x=x,y=error] {data/error_fgreedy_g5.dat};
\addlegendentry{Residual-greedy}

\addplot+[mark=none] table[x=x,y=error] {data/error_fixed_g5.dat};
\addlegendentry{Geometric greedy}
\end{axis}
\end{tikzpicture}

\caption{Comparison of the $\infty$-norm approximation error for $u_\H$ using residual-greedy and geometric greedy collocation points.}
\label{fig:error-comparison-k-g5-500}
\end{figure}
The improved residual performance of the residual-greedy method is further supported by Figure~\ref{fig:k-point-distribution}. The selected centers concentrate in regions that have a stronger influence on the residual, while comparatively less relevant regions are sampled less densely. The fact that the improvement in the residual is not fully reflected in the $\infty$-norm error is not surprising, since for this nonlinear PDE the residual does not directly control the $\infty$-norm error. This is consistent with the linear theory in \cite{bib3}, where the strongest convergence results are also formulated in terms of the residual. 

\begin{figure}[ht]
\centering
\begin{minipage}{0.45\textwidth}
    \includegraphics[
        trim=80 0 80 0,
        clip,
        width=\textwidth
    ]{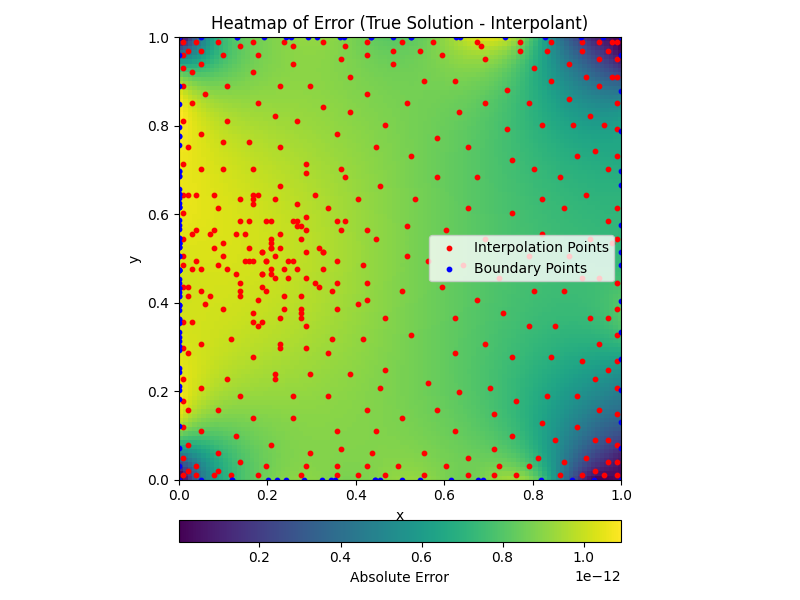}
    \centering a)
\end{minipage}
\hspace{0.05\textwidth}
\begin{minipage}{0.45\textwidth}
    \includegraphics[
        trim=80 0 80 0,
        clip,
        width=\textwidth
    ]{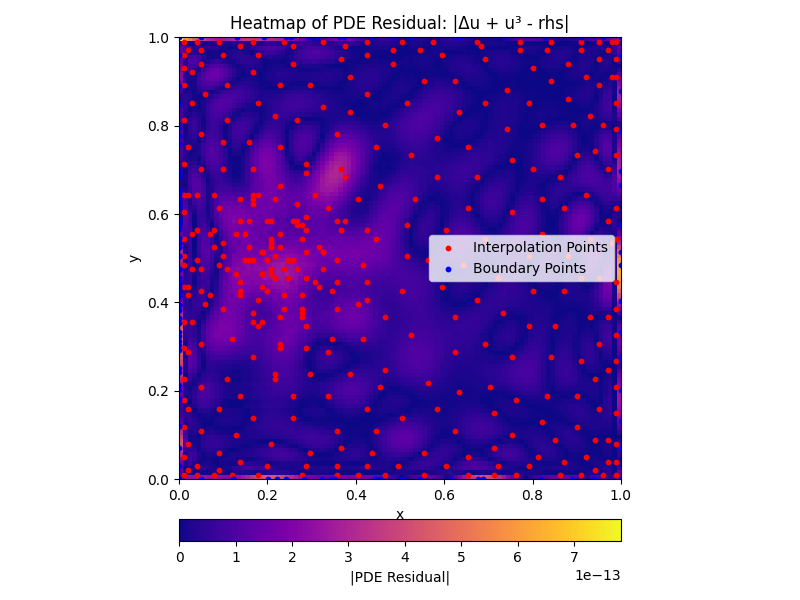}
    \centering b)
\end{minipage}
\caption{Distribution of the collocation points after residual-greedy iterations for the PDE solution $u_\H$. The points are displayed together with a) the error and b) the interior residual.}
\label{fig:k-point-distribution}
\end{figure}
For the test problem with exact solution $u_{\sin}$, the results show an even clearer advantage of the residual-greedy method. As shown in Figure~\ref{fig:residual-comparison-sin-g5-1000}, the residual-greedy approach substantially outperforms the predetermined collocation points with respect to the residual. Although this behaviour is not covered by the convergence theory developed above, it suggests that the method may also be effective for PDEs whose solutions do not lie in the RKHS, and motivates further investigation.

\begin{figure}[ht]
\centering
\begin{tikzpicture}
\begin{axis}[
    width=0.78\textwidth,
    height=45mm,
    xlabel={Number of collocation points},
    ylabel={$\|R_P[u_n]\|_\infty$},
    ymode=log,
    grid=both,
    ytick={1e-6,1e-4,1e-2,1e0},
    minor x tick num=0,
    every axis plot/.append style={line width=1.2pt},
    ymax=1e2,
    ymin=1e-8,
    xmin=0,
    xmax=1000,
    legend style={
        at={(0.98,0.98)},
        anchor=north east,
        legend columns=1
    },
]
\addplot+[mark=none] table[x=x,y=residual] {data/residual_fgreedy_sin_g5.dat};
\addlegendentry{Residual-greedy}

\addplot+[mark=none] table[x=x,y=residual] {data/residual_fixed_sin_g5.dat};
\addlegendentry{Geometric greedy}
\end{axis}
\end{tikzpicture}
\caption{Comparison of the interior residual for residual-greedy and geometric greedy collocation points for the test solution $u_{\sin}$.}
\label{fig:residual-comparison-sin-g5-1000}
\end{figure}
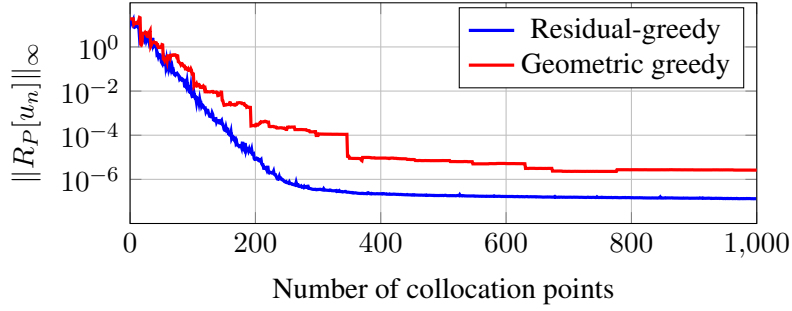
Overall, the experiments support the qualitative convergence results established above. The residual-greedy strategy consistently reduces the PDE residual more effectively than the geometric greedy baseline, whereas the improvement in the $\infty$-norm error appears to depend more strongly on the relation between the residual and the solution error. We expect that this is partly due to the fact that, in our experience, the greedy selection method leads to a more stable linear system. In general, the results are consistent with the theoretical analysis, which guarantees convergence but does not provide convergence rates or a direct residual-to-error estimate. More detailed experiments and further discussion can be found in \cite{bib7}.
\section{Outlook}
We briefly summarize the main contributions of this work. We established a convergence framework for the nonlinear optimal recovery scheme that works directly within the RKHS and does not require uniqueness of the PDE solution. This framework was used to prove RKHS convergence both for collocation sets with vanishing fill distance and for a new residual-greedy selection strategy. Finally, the numerical experiments demonstrated the practical effectiveness of the residual-greedy approach.

A natural next step is to investigate whether convergence rates analogous to those in \cite{bib3} and \cite{bib15} can be obtained for the nonlinear residual-greedy method. The present experiments do not suggest that such results should hold in full generality. Nevertheless, it remains an interesting question whether convergence rates can be established under suitable additional assumptions. A first step in this direction would be to test the method on a broader class of nonlinear PDEs. Further development is also needed for the solution of the finite-dimensional optimal recovery problems. While the Gauss--Newton method used in this paper, following \cite{bib1}, performs well in the experiments considered here, it is unclear how robustly it extends to the full class of PDEs covered by the theoretical framework. The observed experimental success for functions outside the RKHS also suggests further investigation of PDE solutions beyond the RKHS, as has been analyzed in the linear PDE setting in \cite{bib16}. Another direction for future work is to consider operators beyond differential operators. Integral operators provide a natural next case and may help identify a more abstract class of operator equations to which the present approach can be applied.
\vspace{\baselineskip}
\newpage
\begin{acknowledgement}
    Funded by Deutsche Forschungsgemeinschaft (DFG, German Research Foundation) under Project No.  539436611 and Project No. 540080351. The authors used ChatGPT by OpenAI during the preparation of this paper for brainstorming, language revision, improving exposition, and support with numerical implementation. All AI-assisted output was critically reviewed and revised by the authors. The mathematical arguments, proofs, numerical results, and scientific conclusions presented in this paper remain the authors' responsibility.
\end{acknowledgement}

\printbibliography

\begin{appendix}
\section{Auxiliary results}
\begin{lemma}
\label{lem:min-norm-solutions-nonempty}
    Assume the setting of Assumption~\ref{ass:collocation-functionals}. Let
    \begin{align*}
        \mathcal{S}_k
        :=
        \{u\in\H_k(\overline{\Omega}) \mid
        R_P[u](x)=0\ \forall x\in\Omega,\ 
        R_B[u](y)=0\ \forall y\in\partial\Omega\}
    \end{align*}
    be the set of RKHS solutions to \eqref{eq:pde}, and define
    \begin{align*}
        S_k
        :=
        \left\{
        u\in\mathcal{S}_k \mid
        \|u\|_{\H_k(\overline{\Omega})}
        =
        \inf_{v\in\mathcal{S}_k}\|v\|_{\H_k(\overline{\Omega})}
        \right\}.
    \end{align*}
    Then $S_k\neq\emptyset$.
\end{lemma}

\begin{proof}
    By Assumption~\ref{ass:collocation-functionals}, the set $\mathcal{S}_k$ is nonempty.
    We first show that $\mathcal{S}_k$ is weakly closed. Let $(w_j)_{j\in\N}\subset\mathcal{S}_k$ and assume that $w_j\rightharpoonup w$ in $\H_k(\overline{\Omega})$. Fix $x\in\Omega$ and $q\in[Q]$. By Assumption~\ref{ass:collocation-functionals}, the map $u\mapsto L_q^P u(x)$ is a continuous linear functional on $\H_k(\overline{\Omega})$. Hence $L_q^P w_j(x)\to L_q^P w(x)$. Since this holds for every $q\in[Q]$, and since $\overline P$ is continuous, we obtain $R_P[w_j](x)\to R_P[w](x)$. Since $w_j\in\mathcal{S}_k$, we have $R_P[w_j](x)=0$ for all $j$, and therefore $R_P[w](x)=0$. The same argument, using the boundary operators and the continuity of $\overline B$, gives $R_B[w](y)=0$ for all $y\in\partial\Omega$. Thus $w\in\mathcal{S}_k$, and $\mathcal{S}_k$ is weakly closed.

    Set $m:=\inf_{u\in\mathcal{S}_k}\|u\|_{\H_k(\overline{\Omega})}$. Since $\mathcal{S}_k\neq\emptyset$, there exists a minimizing sequence $(u_j)_{j\in\N}\subset\mathcal{S}_k$ such that $\|u_j\|_{\H_k(\overline{\Omega})}\to m$. In particular, $(u_j)_{j\in\N}$ is bounded in $\H_k(\overline{\Omega})$. Since Hilbert spaces are reflexive, there exists a subsequence, again denoted by $(u_j)_{j\in\N}$, and some $u^\star\in\H_k(\overline{\Omega})$ such that $u_j\rightharpoonup u^\star$ in $\H_k(\overline{\Omega})$. By weak closedness of $\mathcal{S}_k$, we have $u^\star\in\mathcal{S}_k$. Moreover, the norm is weakly lower semicontinuous, hence
    \begin{align*}
        \|u^\star\|_{\H_k(\overline{\Omega})}
        \leq
        \liminf_{j\to\infty}
        \|u_j\|_{\H_k(\overline{\Omega})}
        =
        m.
    \end{align*}
    Since $u^\star\in\mathcal{S}_k$, the definition of $m$ also gives $m\leq \|u^\star\|_{\H_k(\overline{\Omega})}$. Therefore $\|u^\star\|_{\H_k(\overline{\Omega})}=m$, and hence $u^\star\in S_k$. Thus $S_k\neq\emptyset$.
\end{proof}

\begin{lemma}
\label{lem:equicontinuity}
Assume the setting of Assumption~\ref{ass:collocation-functionals}, and let
$U\subset\H_k(\overline{\Omega})$ be bounded.

For $u\in U$, define
\begin{align*}
    P(u)(x)
    &:=
    \overline{P}\bigl(
        L_1^P u(x),\ldots,L_Q^P u(x)
    \bigr),
    \qquad x\in\overline{\Omega},\\
    B(u)(x)
    &:=
    \overline{B}\bigl(
        L_1^B u(x),\ldots,L_R^B u(x)
    \bigr),
    \qquad x\in\partial\Omega .
\end{align*}
Then the families
\begin{align*}
    \{P(u):u\in U\}
    \qquad\text{and}\qquad
    \{B(u):u\in U\}
\end{align*}
are uniformly equicontinuous on $\overline{\Omega}$ and $\partial\Omega$,
respectively.
\end{lemma}
\begin{proof}
For $u\in\H_k(\overline{\Omega})$, set $L^P(u)(x):=(L_1^P u(x),\ldots,L_Q^P u(x))$ for $x\in\overline{\Omega}$ and $L^B(u)(x):=(L_1^B u(x),\ldots,L_R^B u(x))$ for $x\in\partial\Omega$. We prove the assertion for $P$; the proof for $B$ is analogous. For $i\in[Q]$, write $L_i^P=\sum_{|\alpha|\leq \ell}s_i^\alpha\partial^\alpha$, with unused coefficients set equal to zero, and define $\Phi_i^P(x):=\sum_{|\alpha|\leq \ell}s_i^\alpha\partial_2^\alpha k(\cdot,x)$ for $x\in\overline{\Omega}$. By the derivative-reproducing property, $L_i^P u(x)=\langle u,\Phi_i^P(x)\rangle_{\H_k(\overline{\Omega})}$. Since $U$ is bounded, there exists $C>0$ such that $\|u\|_{\H_k(\overline{\Omega})}\leq C$ for all $u\in U$. Hence, for $x,\overline{x}\in\overline{\Omega}$ and $u\in U$,
\begin{align*}
    \|L^P(u)(x)-L^P(u)(\overline{x})\|_2^2
    &\leq
    C^2\sum_{i=1}^Q
    \|\Phi_i^P(x)-\Phi_i^P(\overline{x})\|_{\H_k(\overline{\Omega})}^2 .
\end{align*}
Since $k\in C^{2\ell}(\overline{\Omega}\times\overline{\Omega})$, each map $\Phi_i^P:\overline{\Omega}\to\H_k(\overline{\Omega})$ is continuous, hence uniformly continuous on the compact set $\overline{\Omega}$. It follows that $\{L^P(u):u\in U\}$ is uniformly equicontinuous on $\overline{\Omega}$. Moreover, for every $x\in\overline{\Omega}$ and $u\in U$,
\begin{align*}
    \|L^P(u)(x)\|_2^2
    &\leq
    C^2\sum_{i=1}^Q
    \|\Phi_i^P(x)\|_{\H_k(\overline{\Omega})}^2 .
\end{align*}
The right-hand side is bounded uniformly in $x$, since the functions $\Phi_i^P$ are continuous and $\overline{\Omega}$ is compact. Thus there exists $\overline{C}>0$ such that $\|L^P(u)(x)\|_2\leq\overline{C}$ for all $u\in U$ and $x\in\overline{\Omega}$. Hence all values $L^P(u)(x)$ lie in the compact ball $\overline{B}_{\overline{C}}(0)\subset\mathbb{R}^Q$. Since $\overline{P}$ is continuous, it is uniformly continuous on this ball. Combining this with the uniform equicontinuity of $\{L^P(u):u\in U\}$ shows that $\{P(u):u\in U\}$ is uniformly equicontinuous on $\overline{\Omega}$.

For the boundary operators, write $L_j^B=\sum_{|\alpha|\leq \ell}t_j^\alpha\partial^\alpha$ and define $\Phi_j^B(x):=\sum_{|\alpha|\leq \ell}t_j^\alpha\partial_2^\alpha k(\cdot,x)$ for $x\in\partial\Omega$. Repeating the preceding argument on the compact set $\partial\Omega$ gives uniform boundedness and uniform equicontinuity of $\{L^B(u):u\in U\}$, and the uniform continuity of $\overline{B}$ on the resulting compact range yields the uniform equicontinuity of $\{B(u):u\in U\}$ on $\partial\Omega$.
\end{proof}
\end{appendix}
\end{document}